\documentclass{article}%
\usepackage{amsmath}
\usepackage{amsfonts}
\usepackage{amssymb}
\usepackage{graphicx}%
\newtheorem{theorem}{Theorem}

\newtheorem{definition}[theorem]{Definition}
\newtheorem{example}[theorem]{Example}

\newtheorem{remark}[theorem]{Remark}

\newenvironment{proof}[1][Proof]{\noindent\textbf{#1.} }{\ \rule{0.5em}{0.5em}}
\begin{document}

\title{Transmutations, $L$-bases and complete families of solutions of the stationary
Schr\"{o}dinger equation in the plane}
\author{Hugo M. Campos, Vladislav V. Kravchenko and Sergii M. Torba\\{\small Department of Mathematics, CINVESTAV del IPN, Unidad Queretaro, }\\{\small Libramiento Norponiente No. 2000, Fracc. Real de Juriquilla,
Queretaro, }\\{\small Qro. C.P. 76230 MEXICO e-mail:
vkravchenko@qro.cinvestav.mx\thanks{Research was supported by CONACYT, Mexico.
Hugo Campos additionally acknowledges the support by FCT, Portugal. Research
of the third-named author was supported by DFFD, Ukraine (GP/F32/030) and by
SNSF, Switzerland (JRP IZ73Z0 of SCOPES 2009--2012). }}}
\maketitle

\begin{abstract}
An $L$-basis associated to a linear second-order ordinary differential
operator $L$ is an infinite sequence of functions $\left\{  \varphi
_{k}\right\}  _{k=0}^{\infty}$ such that $L\varphi_{k}=0$ for $k=0,1$,
$L\varphi_{k}=k(k-1)\varphi_{k-2}$, for $k=2,3,\ldots$ and all $\varphi_{k}$
satisfy certain prescribed initial conditions. We study the transmutation
operators related to $L$ in terms of the transformation of powers of the
independent variable $\left\{  (x-x_{0})^{k}\right\}  _{k=0}^{\infty}$ to the
elements of the $L$-basis and establish a precise form of the transmutation
operator realizing this transformation. We use this transmutation operator to
establish a completeness of an infinite system of solutions of the stationary
Schr\"{o}dinger equation from a certain class. The system of solutions is
obtained as an application of the theory of bicomplex pseudoanalytic functions
and its completeness was a long sought result. Its use for constructing
reproducing kernels and solving boundary and eigenvalue problems has been
considered even without the required completeness justification. The obtained
result on the completeness opens the way for further development and
application of the tools of pseudoanalytic function theory.

\end{abstract}

\section{Introduction}

Transmutation operators are a widely used tool in the theory of linear
differential equations (see, e.g., \cite{Gilbert}, \cite{Carroll},
\cite{LevitanInverse}, \cite{Marchenko}, \cite{Trimeche} and the recent review
\cite{Sitnik}). It is well known that under certain quite general conditions
the transmutation operator transmuting the operator $A=-\frac{d^{2}}{dx^{2}%
}+q(x)$ into $B=-\frac{d^{2}}{dx^{2}}$ is a Volterra integral operator with
good properties. Its kernel can be obtained as a solution of the Goursat
problem for the Klein-Gordon equation with the variable coefficient. In the
book \cite{Fage} another approach to the transmutation was developed. It was
shown that to every (regular) linear second-order ordinary differential
operator $L$ one can associate a linear space spanned on a so-called $L$-basis
-- an infinite family of functions $\left\{  \varphi_{k}\right\}
_{k=0}^{\infty}$ such that $L\varphi_{k}=0$ for $k=0,1$, $L\varphi
_{k}=k(k-1)\varphi_{k-2}$, for $k=2,3,\ldots$ and all $\varphi_{k}$ satisfy
certain prescribed initial conditions. Then the operator of transmutation was
introduced as an operation transforming functions from one such linear space
corresponding to a certain operator $L$ to functions from another linear space
corresponding to another operator $M$, and the transformation consists in
substituting the $L$-basis with the $M$-basis preserving the same coefficients
in the expansion.

In the present work we find out how the canonical Volterra integral
transmutation operator acts on powers of the independent variable $x^{k}$
(which represent a basis associated with the operator $\frac{d^{2}}{dx^{2}}$),
introduce a parametrized family of transmutation operators and construct a
transmutation operator which transforms the powers $x^{k}$  into the functions
$\varphi_{k}$ from the $L$-basis. We prove that it is indeed a transmutation
and can be written in the form of a Volterra integral operator.

We apply this result to prove the completeness of certain families of
solutions of linear two-dimensional elliptic equations with variable
complex-valued coefficients. These families of solutions were obtained earlier
\cite{KrAntonio}, \cite{APFT} as scalar parts of bicomplex pseudoanalytic
formal powers and used for solving boundary value and eigenvalue problems
\cite{CCK}. Nevertheless no result on their completeness even in simplest
cases was known due to profound differences between complex and bicomplex
pseudoanalytic function theories and inapplicability of many classical results
and techniques in the bicomplex situation. The use of the constructed
transmutation operators and their extremely fortunate transformation
properties regarding the $L$-bases allow us to observe that the infinite
families of solutions mentioned above are nothing but the transmuted harmonic
polynomials. Using their well known completeness properties together with the
properties of the Volterra integral transmutation operators we obtained
several results on the completeness of families of solutions for equations
with variable complex-valued coefficients.

In Section 2 we introduce the definition and some basic facts about
transmutations together with an example which we constructed for illustrating
some results of the present work. In Section 3 we introduce the $L$-basis as a
system of recursive integrals. In Section 4 we study the action of the
transmutation operators on the recursive integrals and construct the
transmutation operator which transforms powers of $x$ into functions of the
$L$-basis. In Section 5 we introduce several definitions and results from the
recently developed bicomplex pseudoanalytic function theory and explain its
relation to linear second-order elliptic equations with variable
complex-valued coefficients. In Section 6 we construct infinite families of
solutions for a class of such equations and show that they are images of
harmonic polynomials under the transmutation operator. We use this fact to
prove their completeness under certain additional conditions. Section 7
contains some concluding remarks.

\section{Transmutation operators for Sturm-Liouville
equations\label{SectTransmSL}}

According to the definition given by Levitan \cite{LevitanInverse}, let $E$ be
a linear topological space, $A$ and $B$ be linear operators: $E\rightarrow E$.
Let $E_{1}$ and $E_{2\text{ }}$be closed subspaces of $E$.

\begin{definition}
A linear invertible operator $T$ defined on the whole $E$ and acting from
$E_{1}$ to $E_{2}$ is called a transmutation operator for the pair of
operators $A$ and $B$ if it fulfills the following two conditions.
\end{definition}

1. \textit{Both the operator }$T$\textit{ and its inverse }$T^{-1}$\textit{
are continuous in }$E$;

2. \textit{The following operator equality is valid}%

\begin{equation}
AT=TB \label{ATTB}%
\end{equation}

\textit{or which is the same}%
\[
A=TBT^{-1}.
\]
\ \ Our main interest concerns the situation when $A=-\frac{d^{2}}{dx^{2}%
}+q(x)$, $B=-\frac{d^{2}}{dx^{2}}$, \ and $q$ is a continuous complex-valued
function. Hence for our purposes it will be sufficient to consider the
functional space $E=C^{2}[a,b]$ with the topology of uniform convergency. For
simplicity we will assume that the interval is symmetric with respect to the
origin, thus $E=C^{2}[-a,a]$.

An operator of transmutation for such $A$ and $B$ can be realized in the form
(see, e.g., \cite{LevitanInverse} and \cite{Marchenko})%

\begin{equation}
Tu(x)=u(x)+\int_{-x}^{x}K(x,t)u(t)dt \label{T}%
\end{equation}
where $K(x,t)$ is a unique solution of the Goursat problem%

\begin{equation}
\left( \frac{\partial^{2}}{\partial x^{2}}-q(x)\right) K(x,t)=\frac
{\partial^{2}}{\partial t^{2}}K(x,t), \label{Goursat1}%
\end{equation}%
\begin{equation}
K(x,x)=\frac{1}{2}\int_{0}^{x}q(s)ds,\qquad K(x,-x)=0. \label{Goursat2}%
\end{equation}

An important property of this transmutation operator consists in the way how
it maps solutions of the equation%

\begin{equation}
v^{\prime\prime}+\omega^{2}v=0 \label{SLomega1}%
\end{equation}
into solutions of the equation%

\begin{equation}
u^{\prime\prime}-q(x)u+\omega^{2}u=0 \label{SLomega2}%
\end{equation}
where $\omega$ is a complex number. Denote by $e_{0}(i\omega,x)$ the solution
of (\ref{SLomega2}) satisfying the initial conditions%

\begin{equation}
e_{0}(i\omega,0)=1\qquad\text{and}\qquad e_{0}^{\prime}(i\omega,0)=i\omega.
\label{initcond}%
\end{equation}
The subindex "$0$" indicates that the initial conditions correspond to the
point $x=0$ and the letter "$e$" reminds us that the initial values coincide
with the\ initial values of the function $e^{i\omega x}$.

The transmutation operator (\ref{T}) maps $e^{i\omega x}$ into $e_{0}%
(i\omega,x)$,
\begin{equation}
e_{0}(i\omega,x)=T[e^{i\omega x}] \label{e0=Te}%
\end{equation}
(see \cite[Theorem 1.2.1]{Marchenko}).

Following \cite{Marchenko} we introduce the notations%

\[
K(x,t;h)=h+K(x,t)+K(x,-t)+h\int_{t}^{x}\{K(x,\xi)-K(x,-\xi)\}d\xi
\]
where $h$ is a complex number, and
\[
K(x,t;\infty)=K(x,t)-K(x,-t).
\]

\begin{theorem}
\cite{Marchenko}\textit{ Solutions }$c(\omega,x;h)$\textit{ and }%
$s(\omega,x;\infty)$\textit{ of equation (}\ref{SLomega2}\textit{) satisfying
the initial conditions}%
\[
c(\omega,0;h)=1,\text{ \ \ \ \ \ \ \ \ \ \ \ \ }c_{x}^{\prime}(\omega,0;h)=h
\]%
\[
s(\omega,0;\infty)=0,\text{ \ \ \ \ \ \ \ \ \ \ \ }s_{x}^{\prime}%
(\omega,0;\infty)=1
\]
\textit{can be represented in the form}%
\begin{equation}
\mathit{\ }c(\omega,x;h)=\cos\omega x+\int_{0}^{x}K(x,t;h)\cos\omega t\,dt
\label{c cos}%
\end{equation}
\textit{and}%
\begin{equation}
s(\omega,x;\infty)=\frac{\sin\omega x}{\omega}+\int_{0}^{x}K(x,t;\infty
)\frac{\sin\omega t}{\omega}\,dt. \label{s sin}%
\end{equation}

\end{theorem}

The operators%
\[
T_{c}u(x)=u(x)+\int_{0}^{x}K(x,t;h)u(t)dt
\]
and%
\[
T_{s}u(x)=u(x)+\int_{0}^{x}K(x,t;\infty)u(t)dt
\]
are not transmutations on the whole space $C^{2}[-a,a]$, they even do not map
all solutions of (\ref{SLomega1}) into solutions of (\ref{SLomega2}). For
example, as we show below%

\[
\left(  -\frac{d^{2}}{dx^{2}}+q(x)\right)  T_{s}[1]\neq T_{s}\left[
-\frac{d^{2}}{dx^{2}}(1)\right]  =0
\]
when $q$ is constant.

\begin{example}
\label{ExampleTransmut} Transmutation operator for operators $A:=\frac{d^{2}%
}{dx^{2}}+c$, $c$ is a constant, and $B:=\frac{d^{2}}{dx^{2}}$. According to
\cite[(1.2.25), (1.2.26)]{Marchenko}, finding the kernel of transmutation
operator is equivalent to finding the function $H(s,t)=K(s+t,s-t)$, satisfying
Goursat problem
\[
\frac{\partial^{2}H(s,t)}{\partial s\partial t}=-cH(s,t),\quad H(s,0)=-\frac
{cs}{2},\quad H(0,t)=0.
\]
The solution of this problem is given by \cite[(4.85)]{Garab}
\[
H(s,t)=-\frac{c}{2}\int_{0}^{s}J_{0}\bigl(2\sqrt{ct(s-\xi)}\bigr)\,d\xi
=-\frac{\sqrt{cst}J_{1}(2\sqrt{cst})}{2t},
\]
where $J_{0}$ and $J_{1}$ are Bessel functions of first kind, and the formula
is valid even if the radicand is negative. Hence,
\begin{equation}
K(x,y)=H\left(  \frac{x+y}{2},\frac{x-y}{2}\right)  =-\frac{1}{2}\frac
{\sqrt{c(x^{2}-y^{2})}J_{1}\bigl(\sqrt{c(x^{2}-y^{2})}\bigr)}{x-y}.
\label{ExampleKernel}%
\end{equation}
From \eqref{ExampleKernel} we get the `sine' kernel
\begin{equation}
K(x,t;\infty)=-\frac{t\sqrt{c(x^{2}-t^{2})}J_{1}\bigl(\sqrt{c(x^{2}-t^{2}%
)}\bigr)}{x^{2}-t^{2}}, \label{ExampleSineKernel}%
\end{equation}
and can check the above statement about operator $T_{s}$,
\begin{gather*}
T_{s}[1](x)=1-\int_{0}^{x}\frac{t\sqrt{c(x^{2}-t^{2})}J_{1}\bigl(\sqrt
{c(x^{2}-t^{2})}\bigr)}{x^{2}-t^{2}}\,dt=J_{0}(x\sqrt{c}),\\
\left(  \frac{d^{2}}{dx^{2}}+c\right)  T_{s}[1]=\frac{\sqrt{c}J_{1}(x\sqrt
{c})}{x}\neq0.
\end{gather*}

\end{example}

\section{A complete system of recursive integrals}

Let $f\in C^{2}(a,b)\cap C^{1}[a,b]$ be a complex valued function and
$f(x)\neq0$ for any $x\in\lbrack a,b]$. The interval $(a,b)$ is supposed to be
finite. Let us consider the following auxiliary functions%
\begin{equation}
\widetilde{X}^{(0)}(x)\equiv X^{(0)}(x)\equiv1, \label{X1}%
\end{equation}%
\begin{equation}
\widetilde{X}^{(n)}(x)=n%
{\displaystyle\int\limits_{x_{0}}^{x}}
\widetilde{X}^{(n-1)}(s)\left(  f^{2}(s)\right)  ^{(-1)^{n-1}}\,\mathrm{d}s,
\label{X2}%
\end{equation}%
\begin{equation}
X^{(n)}(x)=n%
{\displaystyle\int\limits_{x_{0}}^{x}}
X^{(n-1)}(s)\left(  f^{2}(s)\right)  ^{(-1)^{n}}\,\mathrm{d}s, \label{X3}%
\end{equation}
where $x_{0}$ is an arbitrary fixed point in $[a,b]$. We introduce the
infinite system of functions $\left\{  \varphi_{k}\right\}  _{k=0}^{\infty}$
defined as follows%

\begin{equation}
\varphi_{k}(x)=\left\{
\begin{tabular}
[c]{ll}%
$f(x)X^{(k)}(x)$, & $k$ \text{odd,}\\
$f(x)\widetilde{X}^{(k)}(x)$, & $k$ \text{even,}%
\end{tabular}
\ \ \ \ \ \ \ \ \ \ \right.  \ \label{phik}%
\end{equation}
where the definition of $X^{(k)}$ and $\widetilde{X}^{(k)}$ is given by
(\ref{X1})-(\ref{X3}) with $x_{0}$ being an arbitrary point of the interval
$[a,b]$.

\begin{example}
\label{ExamplePoly}Let $f\equiv1$, $a=0$, $b=1$. Then it is easy to see that
choosing $x_{0}=0$ we have $\varphi_{k}(x)=x^{k}$, $k\in\mathbb{N}_{0}$ where
by $\mathbb{N}_{0}$ we denote the set of non-negative integers.
\end{example}

In \cite{KrCMA2011} it was shown that the system $\left\{  \varphi
_{k}\right\}  _{k=0}^{\infty}$ is complete in $L_{2}(a,b)$ and in \cite{KMoT}
its completeness in the space of piecewise differentiable functions with
respect to the maximum norm was obtained and the corresponding series
expansions in terms of the functions $\varphi_{k}$ were studied.

The system (\ref{phik}) is closely related to the notion of the $L$-basis
introduced and studied in \cite{Fage}. Here the letter $L$ corresponds to a
linear ordinary differential operator. This becomes more transparent from the
following result obtained in \cite{KrCV08} (for additional details and simpler
proof see \cite{APFT} and \cite{KrPorter2010}) establishing the relation of
the system of functions $\left\{  \varphi_{k}\right\}  _{k=0}^{\infty}$ to
Sturm-Liouville equations.

\begin{theorem}
\label{ThGenSolSturmLiouville}\cite{KrCV08} Let $q$ be a continuous complex
valued function of an independent real variable $x\in\lbrack a,b],$ $\lambda$
be an arbitrary complex number. Suppose there exists a solution $f$ of the
equation
\begin{equation}
f^{\prime\prime}-qf=0\label{SLhom}%
\end{equation}
on $(a,b)$ such that $f\in C^{2}[a,b]$ and $f\neq0$ on $[a,b]$. Then the
general solution of the equation
\begin{equation}
u^{\prime\prime}-qu=\lambda u\label{SLlambda}%
\end{equation}
on $(a,b)$ has the form%
\[
u=c_{1}u_{1}+c_{2}u_{2}%
\]
where $c_{1}$ and $c_{2}$ are arbitrary complex constants,
\begin{equation}
u_{1}=%
{\displaystyle\sum\limits_{k=0}^{\infty}}
\frac{\lambda^{k}}{(2k)!}\varphi_{2k}\quad\text{and}\quad u_{2}=%
{\displaystyle\sum\limits_{k=0}^{\infty}}
\frac{\lambda^{k}}{(2k+1)!}\varphi_{2k+1}\label{u1u2}%
\end{equation}
and both series converge uniformly on $[a,b]$.
\end{theorem}

\begin{remark}
\label{RemInitialValues}It is easy to see that by definition the solutions
$u_{1}$ and $u_{2}$ satisfy the following initial conditions
\begin{equation}
u_{1}(x_{0})=f(x_{0}),\qquad u_{1}^{\prime}(x_{0})=f^{\prime}(x_{0}),
\label{initial1}%
\end{equation}%
\begin{equation}
u_{2}(x_{0})=0,\qquad u_{2}^{\prime}(x_{0})=1/f(x_{0}). \label{initial2}%
\end{equation}

\end{remark}

\section{Transmutations and systems of recursive integrals}

Let us obtain the expansion of the solution $e_{0}(i\omega,x)$ \ from Section
\ref{SectTransmSL} in terms of the functions $\varphi_{k}$. We suppose that
$f$ is a solution of (\ref{SLhom}) fulfilling the condition of Theorem
\ref{ThGenSolSturmLiouville} on a finite interval $(-a,a)$. We normalize $f$
in such a way that $f(0)=1$ and let $f^{\prime}(0)=h$ where $h$ is some
complex constant. Then according to Remark \ref{RemInitialValues} the
solutions (\ref{u1u2}) of equation (\ref{SLlambda}) have the following initial
values%
\[
u_{1}(0)=1,\qquad u_{1}^{\prime}(0)=h,\qquad u_{2}(0)=0,\qquad u_{2}^{\prime
}(0)=1.
\]
Hence due to (\ref{initcond}) we obtain $e_{0}(i\omega,x)=u_{1}(x)+(i\omega
-h)u_{2}(x)$. From (\ref{e0=Te}) and (\ref{u1u2}) we have the equality%
\[%
{\displaystyle\sum\limits_{k=0}^{\infty}}
\frac{(i\omega)^{2k}}{(2k)!}\varphi_{2k}(x)+(i\omega-h)%
{\displaystyle\sum\limits_{k=0}^{\infty}}
\frac{(i\omega)^{2k}}{(2k+1)!}\varphi_{2k+1}(x)=%
{\displaystyle\sum\limits_{j=0}^{\infty}}
\frac{(i\omega)^{j}x^{j}}{j!}+\int_{-x}^{x}\left(  K(x,t)%
{\displaystyle\sum\limits_{j=0}^{\infty}}
\frac{(i\omega)^{j}t^{j}}{j!}\right)  dt.
\]
As the series under the sign of integral converges uniformly and the kernel
$K(x,t)$ is at least continuously differentiable (for a continuous $q$
\cite{Marchenko}) we obtain the following relation%
\[%
{\displaystyle\sum\limits_{k=0}^{\infty}}
\frac{(i\omega)^{2k}}{(2k)!}\varphi_{2k}(x)+%
{\displaystyle\sum\limits_{k=0}^{\infty}}
\frac{(i\omega)^{2k+1}}{(2k+1)!}\varphi_{2k+1}(x)-h%
{\displaystyle\sum\limits_{k=0}^{\infty}}
\frac{(i\omega)^{2k}}{(2k+1)!}\varphi_{2k+1}(x)=%
{\displaystyle\sum\limits_{j=0}^{\infty}}
\frac{(i\omega)^{j}}{j!}\left(  x^{j}+\int_{-x}^{x}K(x,t)\,t^{j}dt\right)  .
\]
The equality holds for any $\omega$ hence we obtain the termwise relations%
\begin{equation}
\varphi_{k}=T[x^{k}]\text{\quad when }k\text{ is odd} \label{Txk_odd}%
\end{equation}
and
\begin{equation}
\varphi_{k}-\frac{h}{k+1}\varphi_{k+1}=T[x^{k}]\text{\quad when }%
k\in\mathbb{N}_{0}\text{ is even.} \label{Txk_even1}%
\end{equation}
Taking into account the first of these relations the second can be written
also as follows%
\begin{equation}
\varphi_{k}=T\left[  x^{k}+\frac{h}{k+1}x^{k+1}\right]  \text{\quad when }%
k\in\mathbb{N}_{0}\text{ is even.} \label{Txk_even2}%
\end{equation}

Thus, we proved the following statement.

\begin{theorem}
\label{Th Transmutation of Powers K}Let $q$ be a continuous complex valued
function of an independent real variable $x\in\lbrack-a,a]$, and $f$ be a
particular solution of (\ref{SLhom}) such that $f\in C^{2}[-a,a]$, $f\neq0$ on
$[-a,a]$  and normalized as $f(0)=1$. Denote $h:=f^{\prime}(0)\in\mathbb{C}$.
Suppose $T$ is the operator defined by (\ref{T}) where the kernel $K$ is a
solution of the problem (\ref{Goursat1}), (\ref{Goursat2}) and $\varphi_{k}$,
$k\in\mathbb{N}_{0}$ are functions defined by (\ref{phik}). Then equalities
(\ref{Txk_odd})-(\ref{Txk_even2}) hold.
\end{theorem}

Thus, we clarified what is the result of application of the transmutation $T$
to the powers of the independent variable. This is very useful due to the fact
that as a rule the construction of the kernel $K(x,t)$ in a more or less
explicit form up to now is impossible. Our result gives an algorithm for
transmuting functions which can be represented or at least approximated by
finite or infinite polynomials in the situation when $K(x,t)$ is unknown.

\begin{remark}
Let $f$ be the solution of (\ref{SLhom}) satisfying the initial conditions
\begin{equation}
f(0)=1,\quad\text{and}\quad f^{\prime}(0)=0. \label{initcond 1 0}%
\end{equation}
If it does not vanish on $[-a,a]$ then from Theorem
\ref{Th Transmutation of Powers K} we obtain that $\varphi_{k}=T[x^{k}]$ for
any $k\in\mathbb{N}_{0}$. In general, of course there is no guaranty that the
solution with such initial values have no zeros on $[-a,a]$ and hence the
operator $T$ transmutes the powers of $x$ into $\varphi_{k}$ whose
construction is based on the solution $f$ satisfying (\ref{initcond 1 0}) only
in some neighborhood of the origin.
\end{remark}

Similarly to Theorem \ref{Th Transmutation of Powers K} we obtain the
following statement.

\begin{theorem}
\label{Th Transmutation of Powers Tc and Ts} Under the conditions of Theorem
\ref{Th Transmutation of Powers K} the following equalities are valid%
\begin{equation}
\varphi_{k}=T_{c}[x^{k}]\text{\quad when }k\in\mathbb{N}_{0}\text{ is even}
\label{Tc xk}%
\end{equation}
and
\begin{equation}
\varphi_{k}=T_{s}[x^{k}]\text{\quad when }k\in\mathbb{N}\text{ is odd.}
\label{Ts xk}%
\end{equation}

\end{theorem}

\begin{proof}
It is easy to see that $c(\omega,x;h)=u_{1}(x)$ and $s(\omega,x;\infty
)=u_{2}(x)$ where $u_{1}$ and $u_{2}$ are defined by (\ref{u1u2}). From here
and from (\ref{c cos}), (\ref{s sin}) by expanding $\cos\omega x$ and
$\sin\omega x$ into their Taylor series we obtain (\ref{Tc xk}) and
(\ref{Ts xk}).
\end{proof}

Now, for a given nonvanishing solution of (\ref{SLhom}) on $(-a,a)$ satisfying
the initial conditions $f(0)=1$ and $f^{\prime}(0)=h$ where $h\in\mathbb{C}$
and for the corresponding system of functions (\ref{phik}) we construct a
transmutation operator for the pair $\frac{d^{2}}{dx^{2}}$ and $\frac{d^{2}%
}{dx^{2}}-q(x)$ such that $x^{k}$ are transformed into $\varphi_{k}(x)$ on the
whole segment $[-a,a]$ for any $k\in\mathbb{N}_{0}$. For this we introduce the
following projectors acting on any continuous function (defined on $[-a,a]$)
according to the rules $P_{e}f(x)=(f(x)+f(-x))/2$ and $P_{o}%
f(x)=(f(x)-f(-x))/2$. Consider the following operator
\[
\mathbf{T}=T_{c}P_{e}+T_{s}P_{o}.
\]
It is easy to see that by construction\ for an even $k$ we obtain
$\mathbf{T}[x^{k}]=T_{c}P_{e}[x^{k}]=T_{c}[x^{k}]=\varphi_{k}$ and analogously
for an odd $k$, $\mathbf{T}[x^{k}]=\varphi_{k}$ due to (\ref{Ts xk}). Thus,
$\varphi_{k}=\mathbf{T}[x^{k}]$ for any $k\in\mathbb{N}_{0}$. Moreover, the
operator $\mathbf{T}$ can be written as a Volterra operator in a form similar
to (\ref{T}). We have
\begin{equation}
\mathbf{T}u(x)=u(x)+\int_{-x}^{x}\mathbf{K}(x,t;h)u(t)dt \label{Tmain}%
\end{equation}
where
\begin{equation}
\mathbf{K}(x,t;h)=\frac{h}{2}+K(x,t)+\frac{h}{2}\int_{t}^{x}\left(
K(x,s)-K(x,-s)\right)  ds. \label{Kmain}%
\end{equation}
Let us notice that $\mathbf{K}(x,t;0)=K(x,t)$ and that the expression
\[
\mathbf{K}(x,t;h)-\mathbf{K}(x,-t;h)=K(x,t)-K(x,-t)+\frac{h}{2}\int_{-t}%
^{t}\left(  K(x,s)-K(x,-s)\right)  ds=K(x,t)-K(x,-t)
\]
does not depend on $h$. Thus, we obtain a way to compute $\mathbf{K}(x,t;h)$
for any $h$ by a given $\mathbf{K}(x,t;h_{1})$ for some particular value
$h_{1}$.

\begin{theorem}
The integral kernels $\mathbf{K}(x,t;h)$ and $\mathbf{K}(x,t;h_{1})$ are
related by the expression
\begin{equation}
\mathbf{K}(x,t;h)=\frac{h-h_{1}}{2}+\mathbf{K}(x,t;h_{1})+\frac{h-h_{1}}%
{2}\int_{t}^{x}\left(  \mathbf{K}(x,s;h_{1})-\mathbf{K}(x,-s;h_{1})\right)
ds. \label{KmainChangeOfH}%
\end{equation}

\end{theorem}

Let us prove that $\mathbf{T}$ is indeed a transmutation.

\begin{theorem}
\label{Th Transmute}Under the conditions of Theorem
\ref{Th Transmutation of Powers K} let us assume additionally that $q\in
C^{1}[-a,a]$. Then the operator (\ref{Tmain}) with the kernel defined by
(\ref{Kmain}) transforms $x^{k}$ into $\varphi_{k}(x)$ for any $k\in
\mathbb{N}_{0}$ and
\begin{equation}
\left(  -\frac{d^{2}}{dx^{2}}+q(x)\right)  \mathbf{T}[u]=\mathbf{T}\left[
-\frac{d^{2}}{dx^{2}}(u)\right]  \label{TransmutC2}%
\end{equation}
for any $u\in C^{2}[-a,a]$.
\end{theorem}

\begin{proof}
Under the condition $q\in C^{1}[-a,a]$, the kernel $K(x,t)$ in \eqref{T} is
twice continuously differentiable with respect to both $x$ and $t$
\cite[Theorem 1.2.2]{Marchenko}. Hence, the kernel $\mathbf{K}(x,t;h)$ is also
twice continuously differentiable with respect to both $x$ and $t$, and
$C^{2}[-a,a]$ is invariant under the operator $\mathbf{T}$, that is, the
left-hand side of \eqref{TransmutC2} is well defined for all $u\in
C^{2}[-a,a]$.

Since $\left(  -\frac{d^{2}}{dx^{2}}+q(x)\right)  \varphi_{k}=k(k-1)\varphi
_{k-2},\ k\geq2$, and $\left(  -\frac{d^{2}}{dx^{2}}+q(x)\right)  \varphi
_{k}=0,\ k=0,1$ (see \cite{KrCV08}), the equality \eqref{TransmutC2} is valid
for all powers $x^{k}$ and, by linearity, for all polynomials. Let $u\in
C^{2}[-a,a]$. Then, $u^{\prime\prime}\in C[-a,a]$ and by the Weierstrass
theorem there exists a sequence of polynomials $Q_{n}$ such that
\begin{equation}
Q_{n}\rightarrow u^{\prime\prime},\ n\rightarrow\infty\text{ uniformly on
}[-a,a].\label{2derivconv}%
\end{equation}
Integrating \eqref{2derivconv} twice we conclude that the sequence of
polynomials defined by
\[
P_{n}(x)=u(0)+u^{\prime}(0)x+\int_{0}^{x}\int_{0}^{t}Q_{n}(s)\,ds\,dt
\]
is such that%
\[
P_{n}\rightarrow u,\quad P_{n}^{\prime}\rightarrow u^{\prime}\ \text{ and
}\ P_{n}^{\prime\prime}\rightarrow u^{\prime\prime},\ n\rightarrow\infty
\]
uniformly in $[-a,a]$. Since the kernel $\mathbf{K}(x,t;h)$ is twice
continuously differentiable, it is easy to see that also
\[
\mathbf{T}[P_{n}]\rightarrow\mathbf{T}[u],\quad\bigl(\mathbf{T}[P_{n}%
]\bigr)^{\prime}\rightarrow\bigl(\mathbf{T}[u]\bigr)^{\prime}\text{ and
}\bigl(\mathbf{T}[P_{n}]\bigr)^{\prime\prime}\rightarrow\bigl(\mathbf{T}%
[u]\bigr)^{\prime\prime},\ n\rightarrow\infty
\]
uniformly in $[-a,a]$. Therefore,
\[
\left(  -\frac{d^{2}}{dx^{2}}+q(x)\right)  \mathbf{T}[u]=\lim_{n\rightarrow
\infty}\left(  -\frac{d^{2}}{dx^{2}}+q(x)\right)  \mathbf{T}[P_{n}%
]=\lim_{n\rightarrow\infty}\mathbf{T}\left[  -\frac{d^{2}}{dx^{2}}%
(P_{n})\right]  =\mathbf{T}\left[  -\frac{d^{2}}{dx^{2}}(u)\right]  .
\]

\end{proof}

It is possible to give another final part of the proof, revealing some
properties of the adjoint operator $\mathbf{T}^{\ast}$.

\begin{proof}
Consider $u\in C^{2}[-a,a]$ and let $\{p_{n}\}_{n\in\mathbb{N}}$ be a sequence
of polynomials such that $p_{n}\rightarrow u,\ n\rightarrow\infty$. Consider
$f_{n}=\left(  -\frac{d^{2}}{dx^{2}}+q(x)\right)  \mathbf{T}[p_{n}]$. Since
$p_{n}$ is a polynomial, we also have $f_{n}=\mathbf{T}\left[  -\frac{d^{2}%
}{dx^{2}}(p_{n})\right]  $. Let $f=\left(  -\frac{d^{2}}{dx^{2}}+q(x)\right)
\mathbf{T}[u]$, $\tilde{f}=\mathbf{T}\left[  -\frac{d^{2}}{dx^{2}}(u)\right]
$. From now on consider $u$, $p_{n}$, $f_{n}$, $f$, $\tilde{f}$ as elements of
the Hilbert space $L_{2}[-a,a]$. Since $\mathbf{T}$ is the Volterra operator,
it is continuous in $L_{2}[-a,a]$ space. Hence, $\mathbf{T}p_{n}%
\rightarrow\mathbf{T}u,\ n\rightarrow\infty$. Consider any function $\psi\in
C_{0}^{2}[-a,a]$, that is twice continuously differentiable and supported on
some $[\alpha,\beta]\subset(-a,a)$. Then we have
\begin{equation}
(f_{n}-f,\psi)=\left(  \left(  -\frac{d^{2}}{dx^{2}}+q(x)\right)
(\mathbf{T}p_{n}-\mathbf{T}u),\psi\right)  =\left(  \mathbf{T}p_{n}%
-\mathbf{T}u,\left(  -\frac{d^{2}}{dx^{2}}+\overline{q(x)}\right)
\psi\right)  \rightarrow0,\ n\rightarrow\infty. \label{TransmutC2pr1}%
\end{equation}

For the right-hand side of \eqref{TransmutC2}, we need the adjoint operator
$\mathbf{T}^{\ast}$. Since $\mathbf{T}$ is the Volterra operator, its adjoint
is given by the expression \cite[Chapter 3, Example 3.17]{Kato} $\mathbf{T}%
^{\ast}u(x)=u(x)+\int_{-a}^{-|x|}\overline{\mathbf{K}(t,x;h)}u(t)dt+\int
_{|x|}^{a}\overline{\mathbf{K}(t,x;h)}u(t)dt$, from which it is easy to see
that $\mathbf{T}^{\ast}\psi\in C_{0}^{2}[-a,a]$ for any $\psi\in C_{0}%
^{2}[-a,a]$. Hence,
\begin{equation}
\bigl(f_{n}-\tilde{f},\psi\bigr)=\left(  -\frac{d^{2}}{dx^{2}}(p_{n}%
-u),\mathbf{T}^{\ast}\psi\right)  =\left(  p_{n}-u,-\frac{d^{2}}{dx^{2}%
}\left(  \mathbf{T}^{\ast}\psi\right)  \right)  \rightarrow0,\ n\rightarrow
\infty. \label{TransmutC2pr2}%
\end{equation}

It follows from \eqref{TransmutC2pr1} and \eqref{TransmutC2pr2} that
$(f-\tilde{f},\psi)=0$ for any $\psi\in C_{0}^{2}[-a,a]$. Since the set
$C_{0}^{2}[-a,a]$ is dense in $L_{2}[-a,a]$, we have $f=\tilde{f}$ as elements
of $L_{2}[-a,a]$ as well as continuous functions.
\end{proof}

\begin{example}
\label{ExampleTransmut2} Consider the same operators $A$ and $B$ as in Example
\ref{ExampleTransmut}. Then the transmutation operator $\mathbf{T}$ is defined
by the integration kernel \eqref{Kmain}
\[
\mathbf{K}(x,t;h)=-\frac{1}{2}\frac{\sqrt{c(x^{2}-y^{2})}J_{1}\bigl(\sqrt
{c(x^{2}-y^{2})}\bigr)}{x-y}+\frac{h}{2}J_{0}\bigl(\sqrt{c(x^{2}-y^{2}%
)}\bigr)=K(x,t)+\frac{h}{2}J_{0}\bigl(\sqrt{c(x^{2}-y^{2})}\bigr).
\]

If we consider the function $f(x)=e^{i\kappa x},\ \kappa^{2}=c$ as a solution
of $Af=0$ satisfying $f(0)=1$, $f^{\prime}(0)=i\kappa$, then the first four
functions $\varphi_{k}$ are
\[
\varphi_{0}(x)=e^{i\kappa x},\quad\varphi_{1}(x)=\frac{\sin(\kappa x)}{\kappa
},\quad\varphi_{2}(x)=\frac{\kappa xe^{i\kappa x}-\sin(\kappa x)}{i\kappa^{2}%
},\quad\varphi_{3}(x)=\frac{3(\sin(\kappa x)-\kappa x\cos(\kappa x))}%
{\kappa^{3}}.
\]
and from Theorem \ref{Th Transmutation of Powers Tc and Ts} we obtain the
integrals
\[
\varphi_{k}(x)=x^{k}-\frac{\kappa}{2}\int_{-x}^{x}\left(  \frac{\sqrt
{(x^{2}-y^{2})}J_{1}\bigl(\kappa\sqrt{(x^{2}-y^{2})}\bigr)y^{k}}{x-y}%
+iy^{k}J_{0}\bigl(\kappa\sqrt{(x^{2}-y^{2})}\bigr)\right)  \,dy
\]
which validity can be checked numerically.
\end{example}

\section{Bicomplex numbers and pseudoanalytic functions}

Together with the imaginary unit $i$ we consider another imaginary unit $j$,
such that
\begin{equation}
j^{2}=i^{2}=-1\quad\text{and}\quad i\,j=j\,i. \label{laws}%
\end{equation}
We have then two copies of the algebra of complex numbers: $\mathbb{C}%
_{i}:=\left\{  a+ib,\quad\left\{  a,b\right\}  \subset\mathbb{R}\right\}  $
and $\mathbb{C}_{j}:=\left\{  a+jb,\quad\left\{  a,b\right\}  \subset
\mathbb{R}\right\}  $. The expressions of the form $w=u+jv$ where $\left\{
u,v\right\}  \subset\mathbb{C}_{i}$ are called bicomplex numbers. The
conjugation with respect to $j$ we denote as follows $\overline{w}=u-jv$. The
components $u$ and $v$ will be called the scalar and the vector part of $w$
respectively. We will use the notation $u=\operatorname{Sc}w$ and
$v=\operatorname{Vec}w$.

The set of all bicomplex numbers with a natural operation of addition and with
the multiplication defined by the laws (\ref{laws}) represents a commutative
ring with unit. We denote it by $\mathbb{B}$. It contains zero divisors: the
nonzero elements $w$ such that $w\overline{w}=0$. Introducing the pair of
idempotents $P^{+}=\frac{1}{2}(1+ij)$ and $P^{-}=\frac{1}{2}(1-ij)$ ($\left(
P^{\pm}\right)  ^{2}=P^{\pm}$) it is easy to see (e.g., \cite[p. 154]{APFT})
that $w=u+jv$ is a zero divisor if and only if $w=2P^{+}u$ or $w=2P^{-}u$. For
other algebraic properties of bicomplex numbers we refer to
\cite{RochonShapiro}, \cite{RochonTrembl}.

We consider $\mathbb{B}$-valued functions of two real variables $x$ and $y$.
Denote $\overline{\partial}=\frac{1}{2}(\frac{\partial}{\partial x}%
+j\frac{\partial}{\partial y})$ and $\partial=\frac{1}{2}(\frac{\partial
}{\partial x}-j\frac{\partial}{\partial y})$. An equation of the form
\begin{equation}
\overline{\partial}w=aw+b\overline{w},\label{Vekuabic}%
\end{equation}
where $w$, $a$ and $b$ are $\mathbb{B}$-valued functions is called a bicomplex
Vekua equation. When all the involved functions have their values in
$\mathbb{C}_{j}$ only, equation (\ref{Vekuabic}) becomes the well known
complex Vekua equation (see \cite{APFT}, \cite{Vekua}). We will assume that
$w\in C^{1}(\Omega)$ where $\Omega\subset\mathbb{R}^{2}$ is an open domain and
$a$, $b$ are H\"{o}lder continuous in $\Omega$.

When $a\equiv0$ and $b=\frac{\overline{\partial}\phi}{\phi}$ where
$\phi:\overline{\Omega}\rightarrow\mathbb{C}_{i}$ possesses H\"{o}lder
continuous partial derivatives in $\Omega$ and $\phi(x,y)\neq0$,
$\forall(x,y)\in\overline{\Omega}$ we will say that the bicomplex Vekua
equation
\begin{equation}
\overline{\partial}w=\frac{\overline{\partial}\phi}{\phi}\overline
{w}\label{Vekuamain}%
\end{equation}
is a Vekua equation of the main type or the main Vekua equation.

For classical complex Vekua equations Bers introduced \cite{Berskniga} the
notions of a generating pair, generating sequence, formal powers and Taylor
series in formal powers. As was shown in \cite{KrAntonio}, \cite{APFT} the
definition of these notions can be extended onto the bicomplex situation. Here
we briefly recall the main definitions.

\begin{definition}
A pair of $\mathbb{B}$-valued functions $F$ and $G$ possessing H\"{o}lder
continuous partial derivatives in $\Omega$ with respect to the real variables
$x$ and $y$ is said to be a generating pair if it satisfies the inequality%
\begin{equation}
\operatorname{Vec}(\overline{F}G)\neq0\qquad\text{in }\Omega.
\label{condGenPair}%
\end{equation}

\end{definition}

Condition (\ref{condGenPair}) implies that every bicomplex function $w$
defined in a subdomain of $\Omega$ admits the unique representation $w=\phi
F+\psi G$ where the functions $\phi$ and $\psi$ are scalar ($\mathbb{C}_{i}$-valued).

\begin{remark}
When $F\equiv1$ and $G\equiv j$ the corresponding bicomplex Vekua equation is
\begin{equation}
\overline{\partial}w=0, \label{C-Rbic}%
\end{equation}
and its study in fact reduces to the complex analytic function theory. This is
due to the fact that the functions $P^{+}w$ and $P^{-}w$ are necessarily
antiholomorphic and holomorphic respectively. Indeed, application of $P^{+}$
and $P^{-}$ to (\ref{C-Rbic}) gives us
\begin{equation}
\partial_{z}P^{+}w=0\quad\text{and}\quad\partial_{\overline{z}}P^{-}w=0
\label{dP}%
\end{equation}
where $\partial_{z}=\frac{1}{2}(\frac{\partial}{\partial x}-i\frac{\partial
}{\partial y})$ and $\partial_{\overline{z}}=\frac{1}{2}(\frac{\partial
}{\partial x}+i\frac{\partial}{\partial y})$. Moreover, $P^{+}w=P^{+}%
(u+jv)=P^{+}(u-iv)$ and $P^{-}w=P^{-}(u+iv)$. Due to (\ref{dP}) the scalar
functions $w^{+}:=u-iv$ and $w^{-}:=u+iv$ are antiholomorphic and holomorphic
respectively. We stress that $w^{+}$ is not necessarily a complex conjugate of
$w^{-}$ ($u$ and $v$ are $\mathbb{C}_{i}$-valued).
\end{remark}

In general a reduction of the bicomplex Vekua equation (\ref{Vekuabic}) to a
pair of decoupled complex Vekua equations is impossible. Application of
$P^{+}$ and $P^{-}$ to (\ref{Vekuabic}) reduces it to the following system of
equations
\[
\partial_{z}w^{+}=a^{+}w^{+}+b^{+}w^{-}%
\]
and
\[
\partial_{\overline{z}}w^{-}=a^{-}w^{-}+b^{-}w^{+}%
\]
for two complex functions $w^{+}$ and $w^{-}$ with complex coefficients
$a^{\pm}$, $b^{\pm}$.

Assume that $(F,G)$ is a generating pair in a domain $\Omega$.

\begin{definition}
Let the $\mathbb{B}$-valued function $w$ be defined in a neighborhood of
$z_{0}\in\Omega\subset\mathbb{C}_{j}$. In a complete analogy with the complex
case we say that at $z_{0}$ the function $w$ possesses the $(F,G)$-derivative%
\index{(F,G)-derivative}
$\overset{\cdot}{w}(z_{0})$ if the (finite) limit
\begin{equation}
\overset{\cdot}{w}(z_{0})=\lim_{z\rightarrow z_{0}}\frac{w(z)-\lambda
_{0}F(z)-\mu_{0}G(z)}{z-z_{0}} \label{derivative_def}%
\end{equation}
exists where $\lambda_{0}$ and $\mu_{0}$ are the unique scalar constants such
that $w(z_{0})=\lambda_{0}F(z_{0})+\mu_{0}G(z_{0})$.
\end{definition}

Similarly to the complex case (see, e.g., \cite[Chapter 2]{APFT}) it is easy
to show that if $\overset{\cdot}{w}(z_{0})$ exists then at $z_{0}$,
$\overline{\partial}w$ and $\partial w$ exist and equations
\begin{equation}
\overline{\partial}w=a_{(F,G)}w+b_{(F,G)}\overline{w} \label{Vekua_equation}%
\end{equation}
and
\begin{equation}
\overset{\cdot}{w}=\partial w-A_{(F,G)}w-B_{(F,G)}\overline{w}
\label{derivative_with_characteristic}%
\end{equation}
hold, where $a_{(F,G)}$, $b_{(F,G)}$, $A_{(F,G)}$ and $B_{(F,G)}$ are the
\emph{characteristic coefficients}%
\index{characteristic coefficients}
of the pair $(F,G)$ defined by the formulas
\[
a_{(F,G)}=-\frac{\overline{F}\,\overline{\partial}G-\overline{G}%
\,\overline{\partial}F}{F\overline{G}-\overline{F}G},\qquad b_{(F,G)}%
=\frac{F\,\overline{\partial}G-G\,\overline{\partial}F}{F\overline
{G}-\overline{F}G},
\]

\[
A_{(F,G)}=-\frac{\overline{F}\,\partial G-\overline{G}\,\partial F}%
{F\overline{G}-\overline{F}G},\qquad B_{(F,G)}=\frac{F\,\partial G-G\,\partial
F}{F\overline{G}-\overline{F}G}.
\]
Note that $F\overline{G}-\overline{F}G=-2j\operatorname{Vec}(\overline
{F}G)\neq0$.

If $\overline{\partial}w$ and $\partial w$ exist and are continuous in some
neighborhood of $z_{0}$, and if (\ref{Vekua_equation}) holds at $z_{0}$, then
$\overset{\cdot}{w}(z_{0})$ exists, and (\ref{derivative_with_characteristic})
holds. Let us notice that $F$ and $G$ possess $(F,G)$-derivatives,
$\overset{\cdot}{F}\equiv\overset{\cdot}{G}\equiv0$ and the following
equalities are valid which determine the characteristic coefficients uniquely%
\[
\overline{\partial}F=a_{(F,G)}F+b_{(F,G)}\overline{F},\quad\overline{\partial
}G=a_{(F,G)}G+b_{(F,G)}\overline{G},
\]%
\[
\partial F=A_{(F,G)}F+B_{(F,G)}\overline{F},\quad\partial G=A_{(F,G)}%
G+B_{(F,G)}\overline{G}.
\]
If the $(F,G)$-derivative of a $\mathbb{B}$-valued function $w=\phi F+\psi G$
(where the functions $\phi$ and $\psi$ are scalar) exists, besides the form
(\ref{derivative_with_characteristic}) it can also be written as follows
$\overset{\cdot}{w}=\partial\phi\,F+\partial\psi\,G$.

\begin{definition}
\label{DefSuccessor_bi}Let $(F,G)$ and $(F_{1},G_{1})$ -- be two generating
pairs in $\Omega$. $(F_{1},G_{1})$ is called \ successor of $(F,G)$ and
$(F,G)$ is called predecessor of $(F_{1},G_{1})$ if%
\[
a_{(F_{1},G_{1})}=a_{(F,G)}\qquad\text{and}\qquad b_{(F_{1},G_{1})}%
=-B_{(F,G)}\text{.}%
\]

\end{definition}

By analogy with the complex case we have the following statement.

\begin{theorem}
\label{ThBersDer_bi}Let $w$ be a bicomplex $(F,G)$-pseudoanalytic function and
let $(F_{1},G_{1})$ be a successor of $(F,G)$. Then $\overset{\cdot}{w}$ is a
bicomplex $(F_{1},G_{1})$-pseudoanalytic function.
\end{theorem}

\begin{definition}
\label{DefAdjoint_bi}Let $(F,G)$ be a generating pair. Its adjoint generating
pair $(F,G)^{\ast}=(F^{\ast},G^{\ast})$ is defined by the formulas%
\[
F^{\ast}=-\frac{2\overline{F}}{F\overline{G}-\overline{F}G},\qquad G^{\ast
}=\frac{2\overline{G}}{F\overline{G}-\overline{F}G}.
\]

\end{definition}

The $(F,G)$-integral is defined as follows
\[
\int_{\Gamma}Wd_{(F,G)}z=\frac{1}{2}\left(  F(z_{1})\operatorname{Sc}%
\int_{\Gamma}G^{\ast}Wdz+G(z_{1})\operatorname{Sc}\int_{\Gamma}F^{\ast
}Wdz\right)
\]
where $\Gamma$ is a rectifiable curve leading from $z_{0}$ to $z_{1}$.

If $W=\phi F+\psi G$ is a bicomplex $(F,G)$-pseudoanalytic function where
$\phi$ and $\psi$ are complex valued functions then
\begin{equation}
\int_{z_{0}}^{z}\overset{\cdot}{W}d_{(F,G)}z=W(z)-\phi(z_{0})F(z)-\psi
(z_{0})G(z), \label{FGAntD}%
\end{equation}
and as $\overset{\cdot}{F}=\overset{}{\overset{\cdot}{G}=}0$, this integral is
path-independent and represents the $(F,G)$-antiderivative of $\overset{\cdot
}{W}$.

\begin{definition}
\label{DefSeq_bi}A sequence of generating pairs $\left\{  (F_{m}%
,G_{m})\right\}  $, $m=0,\pm1,\pm2,\ldots$ , is called a generating sequence
if $(F_{m+1},G_{m+1})$ is a successor of $(F_{m},G_{m})$. If $(F_{0}%
,G_{0})=(F,G)$, we say that $(F,G)$ is embedded in $\left\{  (F_{m}%
,G_{m})\right\}  $.
\end{definition}

Let $W$ be a bicomplex $(F,G)$-pseudoanalytic function. Using a generating
sequence in which $(F,G)$ is embedded we can define the higher derivatives of
$W$ by the recursion formula%
\[
W^{[0]}=W;\qquad W^{[m+1]}=\frac{d_{(F_{m},G_{m})}W^{[m]}}{dz},\quad
m=1,2,\ldots\text{.}%
\]

\begin{definition}
\label{DefFormalPower_bi}The formal power $Z_{m}^{(0)}(a,z_{0};z)$ with center
at $z_{0}\in\Omega$, coefficient $a$ and exponent $0$ is defined as the linear
combination of the generators $F_{m}$, $G_{m}$ with scalar constant
coefficients $\lambda$, $\mu$ chosen so that $\lambda F_{m}(z_{0})+\mu
G_{m}(z_{0})=a$. The formal powers with exponents $n=0,1,2,\ldots$ are defined
by the recursion formula%
\begin{equation}
Z_{m}^{(n+1)}(a,z_{0};z)=(n+1)\int_{z_{0}}^{z}Z_{m+1}^{(n)}(a,z_{0}%
;\zeta)d_{(F_{m},G_{m})}\zeta. \label{recformulaD}%
\end{equation}

\end{definition}

This definition implies the following properties.

\begin{enumerate}
\item $Z_{m}^{(n)}(a,z_{0};z)$ is an $(F_{m},G_{m})$-pseudoanalytic function
of $z$.

\item If $a^{\prime}$ and $a^{\prime\prime}$ are scalar constants, then
\[
Z_{m}^{(n)}(a^{\prime}+ja^{\prime\prime},z_{0};z)=a^{\prime}Z_{m}%
^{(n)}(1,z_{0};z)+a^{\prime\prime}Z_{m}^{(n)}(j,z_{0};z).
\]

\item The formal powers satisfy the differential relations%
\[
\frac{d_{(F_{m},G_{m})}Z_{m}^{(n)}(a,z_{0};z)}{dz}=nZ_{m+1}^{(n-1)}%
(a,z_{0};z).
\]

\item The asymptotic formulas
\[
Z_{m}^{(n)}(a,z_{0};z)\sim a(z-z_{0})^{n},\quad z\rightarrow z_{0}%
\]
hold.
\end{enumerate}

The case of the main bicomplex Vekua equation is of a special interest due to
the following relation with the stationary Schr\"{o}dinger equation.

\begin{theorem}
\cite{KrJPhys06} Let $W=W_{1}+jW_{2}$ be a solution of the main bicomplex
Vekua equation
\begin{equation}
\overline{\partial}W=\frac{\overline{\partial}\phi}{\phi}\overline{W}%
\quad\text{in }\Omega\label{mainV}%
\end{equation}
where $W_{1}=\operatorname{Sc}W$, $W_{2}=\operatorname{Vec}W$ and the
$\mathbb{C}_{i}$-valued function $\phi$ is a nonvanishing solution of the
equation
\begin{equation}
-\Delta u+q_{1}(x,y)u=0\quad\text{in }\Omega\label{eqq1}%
\end{equation}
where $q_{1}$ is a continuous $\mathbb{C}_{i}$-valued function. Then $W_{1}$
is a solution of (\ref{eqq1}) in $\Omega$ and $W_{2}$ is a solution of the
associated Schr\"{o}dinger equation
\begin{equation}
-\Delta v+q_{2}(x,y)v=0\quad\text{in }\Omega\label{eqq2}%
\end{equation}
where $q_{2}=\frac{1}{8}\frac{\overline{\partial}\phi\,\partial\phi}{\phi^{2}%
}-q_{1}$.
\end{theorem}

We need the following notation. Let $w$ be a $\mathbb{B}$-valued function
defined on a simply connected domain $\Omega$ with $w_{1}=\operatorname{Sc}w$
and $w_{2}=\operatorname{Vec}w$ such that
\begin{equation}
\frac{\partial w_{1}}{\partial y}-\frac{\partial w_{2}}{\partial x}%
=0,\quad\forall\,(x,y)\in\Omega, \label{compcond}%
\end{equation}
and let $\Gamma\subset\Omega$ be a rectifiable curve leading from
$(x_{0},y_{0})$ to $(x,y)$. Then the integral
\[
\overline{A}w(x,y):=2\left(  \int_{\Gamma}w_{1}dx+w_{2}dy\right)
\]
is path-independent, and all $\mathbb{C}_{i}$-valued solutions $\varphi$ of
the equation $\overline{\partial}\varphi=w$ in $\Omega$ have the form
$\varphi(x,y)=\overline{A}w(x,y)+c$ where $c$ is an arbitrary $\mathbb{C}_{i}%
$-constant. In other words the operator $\overline{A}$ denotes the well known
operation for reconstructing the potential function from its gradient.

\begin{theorem}
\cite{KrJPhys06} Let $W_{1}$ be a $\mathbb{C}_{i}$-valued solution of the
Schr\"{o}dinger equation (\ref{eqq1}) in a simply connected domain $\Omega$.
Then a $\mathbb{C}_{i}$-valued solution $W_{2}$ of the associated
Schr\"{o}dinger equation (\ref{eqq2}) such that $W_{1}+jW_{2}$ is a solution
of (\ref{mainV}) in $\Omega$ can be constructed according to the formula
\[
W_{2}=\frac{1}{\phi}\overline{A}\left(  j\,\phi^{2}\,\overline{\partial
}\left(  \frac{W_{1}}{\phi}\right)  \right)  +\frac{c_{1}}{\phi}%
\]
where $c_{1}$ is an arbitrary $\mathbb{C}_{i}$-constant.

Vice versa, given a solution $W_{2}$ of (\ref{eqq2}), the corresponding
\ solution $W_{1}$ of (\ref{eqq1}) such that $W_{1}+jW_{2}$ is a solution of
(\ref{mainV}) has the form
\[
W_{1}=-\phi\overline{A}\left(  \frac{j}{\phi^{2}}\,\overline{\partial}\left(
\phi W_{2}\right)  \right)  +c_{2}\phi
\]
where $c_{2}$ is an arbitrary $\mathbb{C}_{i}$-constant.
\end{theorem}

As was shown in \cite{KrRecentDevelopments} (see also \cite{APFT}) a
generating sequence can be obtained in a closed form, for example, in the case
when $\phi$ has a separable form $\phi=S(s)T(t)$ where $s$ and $t$ are
conjugate harmonic functions and $S$, $T$ are arbitrary twice continuously
differentiable functions. In practical terms this means that whenever the
Schr\"{o}dinger equation (\ref{eqq1}) admits a particular nonvanishing
solution having the form $\phi=f(\xi)\,g(\eta)$ where $(\xi,\eta)$ is one of
the encountered in physics orthogonal coordinate systems in the plane a
generating sequence corresponding to (\ref{mainV}) can be obtained explicitly
\cite[Sect. 4.8]{APFT}. The knowledge of a generating sequence allows one to
construct the formal powers following Definition \ref{DefFormalPower_bi}. This
construction is a simple algorithm which can be quite easily and efficiently
realized numerically \cite{CCK}, \cite{CKR}. Moreover, in the case of a
complex main Vekua equation which in the notations admitted in the present
paper corresponds to the case of $\phi$ being a real-valued function (then the
main bicomplex Vekua equation decouples into two main complex Vekua equations)
the completeness of the system of formal powers was proved \cite{CCK} in the
sense that any pseudoanalytic in $\Omega$ and H\"{o}lder continuous on
$\partial\Omega$ function can be approximated uniformly and arbitrarily
closely by a finite linear combination of the formal powers. The real parts of
the complex pseudoanalytic formal powers represent then a complete system of
solutions of one Schr\"{o}dinger equation meanwhile the imaginary parts give
us a complete system of solutions of the associated Schr\"{o}dinger equation.

In the bicomplex case the system of formal powers is constructed in the same
way as in the complex situation and the system of functions
\begin{equation}
\left\{  \operatorname{Sc}Z_{0}^{(n)}(1,z_{0};z),\quad\operatorname{Sc}%
Z_{0}^{(n)}(j,z_{0};z)\right\}  _{n=0}^{\infty} \label{syst1}%
\end{equation}
is an infinite system of solutions of (\ref{eqq1}). Nevertheless up to now no
result on the completeness of the system of bicomplex formal powers or of the
family of solutions (\ref{syst1}) has been proved. The reason is that such
important basic facts which are in the core of pseudoanalytic function theory
as the similarity principle are not valid for bicomplex Vekua equations.

In the present work in order to establish such completeness for a certain
class of Schr\"{o}dinger equations we implement the transmutation operators.

\section{Complete families of solutions}

First we consider an important special case. Assume that $\phi$ has the form
$\phi(x,y)=f(x)g(y)$ where $f$ and $g$ are arbitrary $\mathbb{C}_{i}$-valued
twice continuously differentiable and nonvanishing functions defined on the
segments $\left[  a_{1},b_{1}\right]  $ and $\left[  a_{2},b_{2}\right]  $
respectively. In this case there exists a periodic generating sequence with a
period two corresponding to the main Vekua equation (\ref{mainV}):%
\[
(F,G)=\left(  fg,\frac{\,j}{fg}\right)  ,\quad(F_{1},G_{1})=\left(  \frac
{g}{f},\,\frac{jf}{g}\right)  ,\quad(F_{2},G_{2})=\left(  F,G\right)
,\quad(F_{3},G_{3})=(F_{1},G_{1}),\ldots,
\]
and the corresponding formal powers admit the following elegant representation
\cite{Berskniga}. We consider the formal powers with the centre at the point
$(x_{0},y_{0})\in\left[  a_{1},b_{1}\right]  \times\left[  a_{2},b_{2}\right]
$. We assume that $f(x_{0})=g(y_{0})=1$ and define the recursive integrals
according to (\ref{X1})-(\ref{X3}) as well as the system of functions
$\left\{  \varphi_{k}\right\}  _{k=0}^{\infty}$ according to (\ref{phik}). In
a similar way we define a system of functions $\left\{  \psi_{k}\right\}
_{k=0}^{\infty}$ corresponding to $g$,
\begin{equation}
\psi_{k}(y)=\left\{
\begin{tabular}
[c]{ll}%
$g(y)Y^{(k)}(y)$, & $k$ \text{odd,}\\
$g(y)\widetilde{Y}^{(k)}(y)$, & $k$ \text{even,}%
\end{tabular}
\ \ \ \ \ \ \ \ \ \ \ \right.  \label{psik}%
\end{equation}
where
\begin{equation}
\widetilde{Y}^{(0)}(y)\equiv Y^{(0)}(y)\equiv1, \label{Y1}%
\end{equation}%
\begin{equation}
\widetilde{Y}^{(n)}(y)=n%
{\displaystyle\int\limits_{y_{0}}^{y}}
\widetilde{Y}^{(n-1)}(s)\left(  g^{2}(s)\right)  ^{(-1)^{n-1}}\,\mathrm{d}s,
\label{Y2}%
\end{equation}%
\begin{equation}
Y^{(n)}(y)=n%
{\displaystyle\int\limits_{y_{0}}^{y}}
Y^{(n-1)}(s)\left(  g^{2}(s)\right)  ^{(-1)^{n}}\,\mathrm{d}s. \label{Y3}%
\end{equation}

Then the formal powers corresponding to (\ref{mainV}) can be defined as
follows. For $\alpha=\alpha^{\prime}+i\alpha^{\prime\prime}$ and $z_{0}%
=x_{0}+jy_{0}$ we have%
\begin{equation}
Z^{(n)}(\alpha,z_{0},z)=f(x)g(y)\operatorname{Sc}\,_{\ast}Z_{{}}^{(n)}%
(\alpha,z_{0},z)+\frac{\,j}{f(x)g(y)}\operatorname{Vec}\,_{\ast}Z_{{}}%
^{(n)}(\alpha,z_{0},z) \label{Zn}%
\end{equation}
where%

\begin{align}
_{\ast}Z^{(n)}(\alpha,z_{0},z)  &  =\alpha^{\prime}%
{\displaystyle\sum\limits_{k=0}^{n}}
\binom{n}{k}X^{(n-k)}j^{k}\widetilde{Y}^{\left(  k\right)  }\text{\ }%
\label{Znodd}\\
&  +j\alpha^{\prime\prime}%
{\displaystyle\sum\limits_{k=0}^{n}}
\binom{n}{k}\widetilde{X}^{(n-k)}j^{k}Y^{\left(  k\right)  }\text{\ \ \ }%
\quad\text{for an odd }n\nonumber
\end{align}
and%

\begin{align}
_{\ast}Z^{(n)}(\alpha,z_{0},z)  &  =\alpha^{\prime}%
{\displaystyle\sum\limits_{k=0}^{n}}
\binom{n}{k}\widetilde{X}^{(n-k)}j^{k}\widetilde{Y}^{\left(  k\right)
}\text{\ }\label{Zneven}\\
&  +j\alpha^{\prime\prime}%
{\displaystyle\sum\limits_{k=0}^{n}}
\binom{n}{k}X^{(n-k)}j^{k}Y^{\left(  k\right)  }\text{\ \ \ \ }\quad\text{for
an even }n.\nonumber
\end{align}

\begin{remark}
Formulae (\ref{Zn})-(\ref{Zneven}) clearly generalize the binomial
representation for the analytic powers $\alpha(z-z_{0})^{n}$. If one chooses
$f\equiv1$ and $g\equiv1$ then $Z^{(n)}(\alpha,z_{0},z)=\alpha(z-z_{0})^{n}$.
\end{remark}

Consider the family of solutions of (\ref{eqq1}) obtained from the scalar
parts of the formal powers (\ref{syst1}). We have
\begin{align*}
\operatorname{Sc}Z^{(n)}(1,z_{0};z)  &  =\phi(x,y)\operatorname{Sc}%
{\displaystyle\sum\limits_{k=0}^{n}}
\binom{n}{k}X^{(n-k)}j^{k}\widetilde{Y}^{\left(  k\right)  }\\
&  =f(x)g(y)%
{\displaystyle\sum\limits_{\text{even }k=0}^{n}}
\binom{n}{k}X^{(n-k)}j^{k}\widetilde{Y}^{\left(  k\right)  }\text{\ \ }%
\quad\text{for an odd }n,
\end{align*}%
\[
\operatorname{Sc}Z^{(n)}(1,z_{0};z)=f(x)g(y)%
{\displaystyle\sum\limits_{\text{even }k=0}^{n}}
\binom{n}{k}\widetilde{X}^{(n-k)}j^{k}\widetilde{Y}^{\left(  k\right)
}\text{\ \ }\quad\text{for an even }n,
\]%
\[
\operatorname{Sc}Z^{(n)}(j,z_{0};z)=f(x)g(y)%
{\displaystyle\sum\limits_{\text{odd }k=1}^{n}}
\binom{n}{k}\widetilde{X}^{(n-k)}j^{k+1}Y^{\left(  k\right)  }\text{\ \ }%
\quad\text{for an odd }n,
\]%
\[
\operatorname{Sc}Z^{(n)}(j,z_{0};z)=f(x)g(y)%
{\displaystyle\sum\limits_{\text{odd }k=1}^{n}}
\binom{n}{k}X^{(n-k)}j^{k+1}Y^{\left(  k\right)  }\text{\ \ }\quad\text{for an
even }n.
\]
Taking into account the definition of the functions $\varphi_{k}$ and
$\psi_{k}$ (equations (\ref{phik}) and (\ref{psik})) it is easy to rewrite the
last four equalities as follows%
\[
\operatorname{Sc}Z^{(n)}(1,z_{0};z)=%
{\displaystyle\sum\limits_{\text{even }k=0}^{n}}
\left(  -1\right)  ^{\frac{k}{2}}\binom{n}{k}\varphi_{n-k}(x)\psi_{k}(y)
\]
and
\[
\operatorname{Sc}Z^{(n)}(j,z_{0};z)=%
{\displaystyle\sum\limits_{\text{odd }k=1}^{n}}
\left(  -1\right)  ^{\frac{k+1}{2}}\binom{n}{k}\varphi_{n-k}(x)\psi_{k}(y).
\]
Thus, for every $n$ we have two nontrivial exact solutions of (\ref{eqq1})
(except for $n=0$ for which we observe that by construction $\operatorname{Sc}%
Z^{(0)}(j,z_{0};z)\equiv0$). This infinite family of solutions can be written
as follows%
\begin{align}
u_{0}(x,y)  &  =f(x)g(y),\label{um0}\\
u_{m}(x,y)  &  =\operatorname{Sc}Z^{(\frac{m+1}{2})}(1,z_{0};z)=%
{\displaystyle\sum\limits_{\text{even }k=0}^{\frac{m+1}{2}}}
\left(  -1\right)  ^{\frac{k}{2}}\binom{\frac{m+1}{2}}{k}\varphi_{\frac
{m+1}{2}-k}(x)\psi_{k}(y)\text{\ }\quad\text{for an odd }m,\label{umodd}\\
u_{m}(x,y)  &  =\operatorname{Sc}Z^{(\frac{m}{2})}(j,z_{0};z)=%
{\displaystyle\sum\limits_{\text{odd }k=1}^{\frac{m}{2}}}
\left(  -1\right)  ^{\frac{k+1}{2}}\binom{\frac{m}{2}}{k}\varphi_{\frac{m}%
{2}-k}(x)\psi_{k}(y)\text{\ }\quad\text{for an even }m. \label{umeven}%
\end{align}

\begin{remark}
In the case when $f\equiv1$ and $g\equiv1$ we obtain that the system $\left\{
u_{m}\right\}  _{m=0}^{\infty}$ is the system of harmonic polynomials
$\left\{  \operatorname{Sc}(z-z_{0})^{n},\quad\operatorname{Sc}\left(
j(z-z_{0})^{n}\right)  \right\}  _{n=0}^{\infty}$. Theorems about its
completeness like the Runge theorem and further related results are well known
(see, e.g., \cite{Colton}, \cite{Suetin}, \cite{Walsh1929}, \cite{Walsh}). It
is convenient to introduce the notation
\begin{align*}
p_{0}(x,y) &  =1,\\
p_{m}(x,y) &  =\operatorname{Sc}(z-z_{0})^{\frac{m+1}{2}}=%
{\displaystyle\sum\limits_{\text{even }k=0}^{\frac{m+1}{2}}}
\left(  -1\right)  ^{\frac{k}{2}}\binom{\frac{m+1}{2}}{k}(x-x_{0})^{\frac
{m+1}{2}-k}(y-y_{0})^{k}\text{\ }\quad\text{for an odd }m,\\
p_{m}(x,y) &  =\operatorname{Sc}\left(  j(z-z_{0})^{\frac{m}{2}}\right)  =%
{\displaystyle\sum\limits_{\text{odd }k=1}^{\frac{m}{2}}}
\left(  -1\right)  ^{\frac{k+1}{2}}\binom{\frac{m}{2}}{k}(x-x_{0})^{\frac
{m}{2}-k}(y-y_{0})^{k}\text{\ }\quad\text{for an even }m.
\end{align*}

\end{remark}

\begin{remark}
Every $u_{m}$ is a result of application of an operator of transmutation to
the corresponding harmonic polynomial $p_{m}$. Indeed, consider for simplicity
$z_{0}=0$ and suppose that $f$ is defined on the segment $\left[  -a,a\right]
$ and $g$ is defined on $\left[  -b,b\right]  $. Both functions are assumed to
be $\mathbb{C}_{i}$-valued twice continuously differentiable and nonvanishing.
Let $\mathbf{T}_{f}$ be the operator $\mathbf{T}$ defined by (\ref{Tmain}) and
$\mathbf{T}_{g}$ be its equivalent associated with the function $g$. That is,%
\begin{equation}
\mathbf{T}_{g}v(y)=v(y)+\int_{-y}^{y}\mathbf{K}_{g}(y,t;g^{\prime}(0))v(t)dt
\label{Tg}%
\end{equation}
where
\[
\mathbf{K}_{g}(y,t;g^{\prime}(0))=\frac{g^{\prime}(0)}{2}+K_{g}(y,t)+\frac
{g^{\prime}(0)}{2}\int_{t}^{y}\left(  K_{g}(y,s)-K_{g}(y,-s)\right)  ds
\]
and $K_{g}$ is a solution of the Goursat problem
\[
\left( \frac{\partial^{2}}{\partial y^{2}}-q_{g}(y)\right) K_{g}%
(y,t)=\frac{\partial^{2}}{\partial t^{2}}K_{g}(y,t),
\]%
\[
K_{g}(y,y)=\frac{1}{2}\int_{0}^{y}q_{g}(s)ds,\qquad K_{g}(y,-y)=0
\]
with $q_{g}:=g^{\prime\prime}/g$. Then
\begin{equation}
u_{m}(x,y)=\mathbf{T}_{f}\mathbf{T}_{g}p_{m}(x,y). \label{umTpm}%
\end{equation}
This relation follows from Theorem \ref{Th Transmute} according to which the
operator $\mathbf{T}_{f}$ maps a $k$-th power of $x$ into $\varphi_{k}(x)$ for
any $k\in\mathbb{N}_{0}$ and similarly the operator $\mathbf{T}_{g}$ maps a
$k$-th power of $y$ into $\psi_{k}(y)$. Moreover, $\mathbf{T}_{f}%
\mathbf{T}_{g}\,p_{m}=\mathbf{T}_{g}\mathbf{T}_{f}\,p_{m}$.
\end{remark}

This observation together with the Runge approximation theorem for harmonic
functions allows us to prove the following Runge-type theorem for the family
of solutions $\left\{  u_{m}\right\}  _{m=0}^{\infty}$.

\begin{theorem}
\label{ThCompl}Let $\Omega\subset\overline{R}=\left[  -a,a\right]
\times\left[  -b,b\right]  $ be a simply connected domain such that together
with any point $(x,y)$ belonging to $\Omega$ the rectangle with the vertices
$(x,y)$, $(-x,y)$, $(x,-y)$ and $(-x,-y)$ also belongs to $\Omega$. Let the
equation
\begin{equation}
\left(  -\Delta+q(x,y)\right)  u(x,y)=0 \label{Schr}%
\end{equation}
in $\Omega$ admit a particular solution of the form $\phi(x,y)=f(x)g(y)$ where
$f$ and $g$ are $\mathbb{C}_{i}$-valued functions, $f\in C^{2}\left[
-a,a\right]  $, $g\in C^{2}\left[  -b,b\right]  $, $f(x)\neq0$, $g(y)\neq0$
for any $x\in\left[  -a,a\right]  $ and $y\in\left[  -b,b\right]  $
(obviously, $q$ has the form $q(x,y)=q_{1}(x)+q_{2}(y)$ with $q_{1}%
=f^{\prime\prime}/f$ and $q_{2}=g^{\prime\prime}/g$). Then any solution $u$ of
(\ref{Schr}) in $\Omega$ can be approximated arbitrarily closely on any
compact subset $K$ of $\Omega$ by a finite linear combination of the functions
$u_{m}$. That is for any $\varepsilon>0$ there exists such a number
$M\in\mathbb{N}$ and such coefficients $\left\{  \alpha_{m}\right\}
_{m=0}^{M}\subset\mathbb{C}_{i}$ that $\left\vert u(x,y)-%
{\displaystyle\sum\limits_{m=0}^{M}}
\alpha_{m}u_{m}(x,y)\right\vert <\varepsilon$ for any point $(x,y)\in K$.
\end{theorem}

\begin{proof}
Let $u$ be a solution of (\ref{Schr}) in $\Omega$. We have \cite{Walsh1929}
that the harmonic function $v:=\mathbf{T}_{f}^{-1}\mathbf{T}_{g}^{-1}u$ can be
approximated in $K$ with respect to the maximum norm by a harmonic polynomial,
$\left\Vert v-P_{M}\right\Vert <\varepsilon_{1}$, where $P_{M}=%
{\displaystyle\sum\limits_{m=0}^{M}}
\alpha_{m}p_{m}$. Now we use the fact that $\mathbf{T}_{f}$ and $\mathbf{T}%
_{g}\,$are bounded Volterra operators possessing bounded inverse operators. We
have
\[
\left\Vert u-\mathbf{T}_{f}\mathbf{T}_{g}P_{M}\right\Vert =\left\Vert
\mathbf{T}_{f}\mathbf{T}_{g}v-\mathbf{T}_{f}\mathbf{T}_{g}P_{M}\right\Vert
\leq\varepsilon_{1}\left\Vert \mathbf{T}_{f}\right\Vert \left\Vert
\mathbf{T}_{g}\right\Vert =\varepsilon
\]
where the norms $\left\Vert \mathbf{T}_{f}\right\Vert $ and $\left\Vert
\mathbf{T}_{g}\right\Vert $ can be estimated in terms of the maximum values of
the corresponding (continuous) kernels $\mathbf{K}_{f}$ and $\mathbf{K}_{g}$.
\end{proof}

The restricting condition on the shape of the domain $\Omega$ is due to the
necessity to have well defined the function $v$ as the image of $u$ under the
application of the transmutation operators. Let one of the functions $f$ or
$g$ be real-valued, for example, $f$. Then to prove the completeness of the
system $\left\{  u_{m}\right\}  _{m=0}^{\infty}$ using a transmutation
operator one can assume the symmetry of the domain only with respect to the
variable $y$.

\begin{theorem}
\label{ThComplLessSymmetry}Let $\Omega\subset\mathbb{R}^{2}$ be a simply
connected domain such that together with any point $(x,y)$ belonging to
$\Omega$ the point $(x,-y)$ and the segment joining $(x,y)$ with $(x,-y)$
belong to $\Omega$ as well. Let the equation (\ref{Schr}) in $\Omega$ admit a
particular solution of the form $\phi(x,y)=f(x)g(y)$ where $f$ is a
real-valued and $g$ is a $\mathbb{C}_{i}$-valued, both twice continuously
differentiable and nonvanishing up to the boundary functions. Then any
solution $u$ of (\ref{Schr}) in $\Omega$ can be approximated arbitrarily
closely on any compact subset $K$ of $\Omega$ by a finite linear combination
of the functions $u_{m}$.
\end{theorem}

\begin{proof}
Consider the equation
\begin{equation}
\left(  -\Delta+q_{1}(x)\right)  v(x,y)=0\quad\text{in }\Omega\label{Schrv}%
\end{equation}
where $q_{1}=f^{\prime\prime}/f$ is real valued. Denote by $\widehat{u}_{m}$
the functions defined by (\ref{um0})--(\ref{umeven}) with $g\equiv1$. Then as
was proved in \cite{CCK} any solution $v$ of (\ref{Schrv}) can be approximated
arbitrarily closely on any compact subset of $\Omega$ by linear combinations
of the functions $\widehat{u}_{m}$. As $u_{m}=\mathbf{T}_{g}\widehat{u}_{m}$,
$m\in\mathbb{N}_{0}$ (where $g$ is the factor in $\phi$ depending on $y$) once
again using the boundedness of $\mathbf{T}_{g}$ and $\mathbf{T}_{g}^{-1}$ we
obtain the completeness of $\left\{  u_{m}\right\}  _{m=0}^{\infty}$.
\end{proof}

Extension of the results of the preceding two theorems onto arbitrary simply
connected domains is possible if equation (\ref{Schr}) has the Runge property
(see, e.g., \cite{Lax1956}, \cite{Bers PDE}, \cite{Colton Analytic PDE}).

\begin{definition}
Equation $Lu=0$ is said to have the Runge approximation property if, whenever,
$\Omega_{1}$ and $\Omega_{2}$ are two simply connected domains, $\Omega_{1}$ a
subset of $\Omega_{2}$, any solution in $\Omega_{1}$ can be approximated
uniformly in compact subsets of $\Omega_{1}$ by a sequence of solutions which
can be extended as solutions to $\Omega_{2}$.
\end{definition}

It is known \cite{Bers PDE}, \cite{Colton Analytic PDE} that the Runge
property in the case of elliptic equations with real-valued coefficients is
equivalent to the (weak) unique continuation property (if every solution of
$Lu=0$ which vanishes in an open set vanishes identically) and is true, e.g.,
for second-order elliptic equations with real-analytic coefficients. Without
going into further details concerning the Runge approximation property which
is beyond the scope of the present work, we prove that if equation
(\ref{Schr}) has this property then the family of solutions $\left\{
u_{m}\right\}  _{m=0}^{\infty}$ is complete in any simply connected domain.

\begin{theorem}
Let equation (\ref{Schr}) in a rectangle $R=(-a,a)\times(-b,b)$ admit a
particular solution of the form $\phi(x,y)=f(x)g(y)$ where $f$ and $g$ are
arbitrary $\mathbb{C}_{i}$-valued twice continuously differentiable and
nonvanishing functions in $[-a,a]$ and $[-b,b]$ respectively. Let
$\Omega\subset R$ be a simply connected domain. Assume equation (\ref{Schr})
has the Runge property. Then any solution $u$ of (\ref{Schr}) in $\Omega$ can
be approximated arbitrarily closely on any compact subset $K$ of $\Omega$ by a
finite linear combination of the functions $u_{m}$ defined by (\ref{um0}%
)--(\ref{umeven}).
\end{theorem}

\begin{proof}
Consider a solution $u$ in $\Omega$ which due to the Runge property can be
approximated on $K$ by a solution $v$ of (\ref{Schr}) in $R$. Due to Theorem
\ref{ThCompl}, $v$ in its turn can be approximated on $K$ by the functions
$u_{m}$ from where we obtain the required approximation of the solution $u$ in
terms of the solutions $u_{m}$.
\end{proof}

In the rest of the present section we show that solutions of (\ref{Schr})
sufficiently smooth up to the boundary of the domain of interest $\Omega$ can
be approximated by functions $u_{m}$ in $\overline{\Omega}$. By $\Sigma
_{\alpha}^{q}$ we denote the linear space of solutions of (\ref{Schr}) in
$\Omega$ satisfying the following regularity requirement $u\in C^{2}%
(\Omega)\cap C^{1+\alpha}(\overline{\Omega})$, $0\leq\alpha\leq1$. This linear
space can be equipped with one of the following scalar products (we assume
that zero is neither a Dirichlet nor a Neumann eigenvalue)
\begin{equation}
<u,v>_{1}=\underset{\partial\Omega}{\int}uv^{\ast}ds\quad\text{and}%
\quad<u,v>_{2}=\underset{\partial\Omega}{\int}\frac{\partial u}{\partial
n}\frac{\partial v^{\ast}}{\partial n}ds\label{otherscalarproducts}%
\end{equation}
where by \textquotedblleft$^{\ast}$\textquotedblright\ we denote the complex
conjugation in $\mathbb{C}_{i}$, and $\frac{\partial}{\partial n}$ is the
outer normal derivative. With the aid of the scalar products
(\ref{otherscalarproducts}) two Bergman-type reproducing kernels
\cite{BergmanShiffer} can be introduced for solving the Dirichlet and Neumann
problems respectively as well as the corresponding eigenvalue problems
\cite{CCK}.

A complete orthonormal system of functions in $\Sigma_{\alpha}^{q}$ with
respect to $<\cdot,\cdot>_{1}$ or $<\cdot,\cdot>_{2}$ allows one to construct
a corresponding \textquotedblleft Dirichlet\textquotedblright\ or
\textquotedblleft Neumann\textquotedblright\ reproducing kernel respectively.
In \cite{CCK} the completeness of the family of solutions $\left\{
u_{m}\right\}  _{m=0}^{\infty}$ of (\ref{Schr}) obtained as real parts of
complex pseudoanalytic formal powers was proved in the case when (\ref{Schr})
admits a particular solution in a separable form $\varphi(s,t)=S(s)T(t)$ where
$S$ and $T$ are arbitrary twice continuously differentiable nonvanishing
real-valued functions, $\Phi=s+it$ is a conformal mapping defined in
$\overline{\Omega}$ and $\Omega$ is a domain bounded by a Jordan curve. In the
same paper it was shown that the completeness of $\left\{  u_{m}\right\}
_{m=0}^{\infty}$ in $\Sigma_{\alpha}^{q}$ in the case when the particular
solution $\varphi$ is complex-valued is an important open problem and its
solution is required not only for solving boundary value problems for
(\ref{Schr}) with a complex-valued coefficient but also for solving spectral
problems for (\ref{Schr}) even in the situation when the coefficient is
real-valued. Here by means of the developed results concerning the
transmutation operators we obtain the completeness of the family of solutions
$\left\{  u_{m}\right\}  _{m=0}^{\infty}$ in $\Sigma_{\alpha}^{q}$ under the
conditions of Theorems \ref{ThCompl} and \ref{ThComplLessSymmetry}.

\begin{theorem}
\label{ThComplUpBoundary}Let $\Omega\subset\overline{R}=\left[  -a,a\right]
\times\left[  -b,b\right]  $ be a simply connected domain such that together
with any point $(x,y)$ belonging to $\Omega$ the rectangle with the vertices
$(x,y)$, $(-x,y)$, $(x,-y)$ and $(-x,-y)$ also belongs to $\Omega$. Let
equation (\ref{Schr}) in $\Omega$ admit a particular solution of the form
$\phi(x,y)=f(x)g(y)$ where $f$ and $g$ are $\mathbb{C}_{i}$-valued functions,
$f\in C^{2}\left[  -a,a\right]  $, $g\in C^{2}\left[  -b,b\right]  $,
$f(x)\neq0$, $g(y)\neq0$ for any $x\in\left[  -a,a\right]  $ and $y\in\left[
-b,b\right]  $. Then the family of solutions $\left\{  u_{m}\right\}
_{m=0}^{\infty}$ is complete in $\Sigma_{\alpha}^{q}$, $\alpha>0$ with respect
to both norms generated by the scalar products (\ref{otherscalarproducts}).
\end{theorem}

\begin{proof}
Let $u\in\Sigma_{\alpha}^{q}$. Consider the harmonic function $v:=\mathbf{T}%
_{f}^{-1}\mathbf{T}_{g}^{-1}u$ which belongs to $\Sigma_{\alpha}^{0}$ due to
the fact that the kernels in both transmutation operators are at least $C^{1}%
$-functions and hence the operator $\mathbf{T}_{f}^{-1}\mathbf{T}_{g}^{-1}$
transforms $\Sigma_{\alpha}^{q}$ into $\Sigma_{\alpha}^{0}$. There exists
(see, e.g., \cite{CCK}) a sequence of harmonic polynomials $P_{M}$ such that
when $M\rightarrow\infty$, $P_{M}\rightarrow v$ uniformly in $\overline
{\Omega}$ together with their first partial derivatives. It is easy to see
that this implies the uniform convergence in $\overline{\Omega}$ of the
sequences $\mathbf{T}_{f}\mathbf{T}_{g}P_{M}=U_{M}\rightarrow u=\mathbf{T}%
_{f}\mathbf{T}_{g}v$, $\frac{\partial U_{M}}{\partial x}\rightarrow
\frac{\partial u}{\partial x}$ and $\frac{\partial U_{M}}{\partial
y}\rightarrow\frac{\partial u}{\partial y}$. Indeed, the uniform convergence
of $U_{M}$ to $u$ follows directly from the boundedness of the transmutation
operators, see the proof of Theorem \ref{ThCompl}, and the verification of the
uniform convergence of the partial derivatives is straightforward. Consider%
\[
\frac{\partial}{\partial x}\mathbf{T}_{f}P_{M}(x,y)=\frac{\partial}{\partial
x}P_{M}(x,y)+\frac{\partial}{\partial x}\int_{-x}^{x}\mathbf{K}_{f}%
(x,t;h)P_{M}(t,y)dt
\]%
\[
=\frac{\partial}{\partial x}P_{M}(x,y)+\int_{-x}^{x}\frac{\partial}{\partial
x}\mathbf{K}_{f}(x,t;h)P_{M}(t,y)dt+\mathbf{K}_{f}(x,x;h)P_{M}(x,y)+\mathbf{K}%
_{f}(x,-x;h)P_{M}(-x,y)
\]
from where due to the uniform convergence of $P_{M}$ to $v$ together with
their partial derivatives and due to the fact that $\frac{\partial}{\partial
x}\mathbf{K}_{f}(x,t;h)$ is continuous, it follows that $\frac{\partial U_{M}%
}{\partial x}\rightarrow\frac{\partial u}{\partial x}$ uniformly in
$\overline{\Omega}$. The proof of the uniform convergence of $\frac{\partial
U_{M}}{\partial y}$ to $\frac{\partial u}{\partial y}$ is analogous.

Now, the completeness of $\left\{  u_{m}\right\}  _{m=0}^{\infty}$ in
$\Sigma_{\alpha}^{q}$ with respect to the norm generated by $<\cdot,\cdot
>_{1}$ follows from the uniform convergence of $U_{M}$ to $u$ in
$\overline{\Omega}$ and hence from the completeness of $\left\{
u_{m}\right\}  _{m=0}^{\infty}$ with respect to the maximum norm in
$\overline{\Omega}$. To verify the completeness of $\left\{  u_{m}\right\}
_{m=0}^{\infty}$ in $\Sigma_{\alpha}^{q}$ with respect to the norm generated
by $<\cdot,\cdot>_{2}$ consider the following chain of relations%
\[
\left\Vert u-U_{M}\right\Vert =\int_{\partial\Omega}\frac{\partial\left(
u-U_{M}\right)  }{\partial n}\frac{\partial\left(  u^{\ast}-U_{M}^{\ast
}\right)  }{\partial n}ds\leq\int_{\partial\Omega}\left\vert \nabla
(u-U_{M})\right\vert ^{2}\,ds\leq L\underset{\partial\Omega}{\sup}\left\vert
\nabla(u-U_{M})\right\vert ^{2}\rightarrow0.
\]

\end{proof}

\begin{theorem}
Let $\Omega\subset\mathbb{R}^{2}$ be a simply connected domain such that
together with any point $(x,y)$ belonging to $\Omega$ the point $(x,-y)$ and
the segment joining $(x,y)$ with $(x,-y)$ belong to $\Omega$ as well. Let
equation (\ref{Schr}) in $\Omega$ admit a particular solution of the form
$\phi(x,y)=f(x)g(y)$ where $f$ is a real-valued and $g$ is a $\mathbb{C}_{i}%
$-valued, both twice continuously differentiable and nonvanishing up to the
boundary functions. Then $\left\{  u_{m}\right\}  _{m=0}^{\infty}$ is complete
in $\Sigma_{\alpha}^{q}$, $\alpha>0$ with respect to both norms generated by
the scalar products (\ref{otherscalarproducts}).
\end{theorem}

\begin{proof}
The first part of the proof is similar to that of Theorem
\ref{ThComplLessSymmetry}. We have that $u_{m}=\mathbf{T}_{g}\widehat{u}_{m}$,
$m\in\mathbb{N}_{0}$ where $\left\{  \widehat{u}_{m}\right\}  _{m=0}^{\infty}$
is complete in $\Sigma_{\alpha}^{q_{1}}$ (see \cite{CCK}) with respect to the
required norms. Then the completeness of $\left\{  u_{m}\right\}
_{m=0}^{\infty}$ is proved analogously to the proof of Theorem
\ref{ThComplUpBoundary}.
\end{proof}

\begin{remark}
In order to prove the completeness of the family of solutions $\left\{
u_{m}\right\}  _{m=0}^{\infty}$ in $\Sigma_{\alpha}^{q}$ with respect to both
norms generated by the scalar products (\ref{otherscalarproducts}) under less
restrictive conditions on the shape of the domain $\Omega$ in fact we need a
result on the existence of an infinite system of solutions of (\ref{Schr}) in
$R\supset\Omega$ and complete in $\Sigma_{\alpha}^{q}(\partial\Omega)$ or in a
maximum norm in $\overline{\Omega}$. If such a system exists then according to
Theorem \ref{ThCompl} every element of it can be approximated arbitrarily
closely by linear combinations of functions $u_{m}$ which would allow one to
prove the completeness of $\left\{  u_{m}\right\}  _{m=0}^{\infty}$. Thus, if
such a complete system exists then $\left\{  u_{m}\right\}  _{m=0}^{\infty}$
is precisely such system. The question on the existence requires further study.
\end{remark}

\section{Conclusions}

Transmutation operators for Sturm-Liouville equations are considered and their
new properties concerning the transformation of certain infinite systems of
functions generated by the Sturm-Liouville operators are presented. These
infinite systems of functions slightly generalize the notion of $L$-bases
\cite{Fage} and play an important role in the theory of linear differential
equations. We show how a transmutation operator can be constructed mapping one
such basis into another and give an application of this result obtaining
several theorems on the completeness of certain families of solutions of
two-dimensional stationary Schr\"{o}dinger equations which are obtained as
scalar parts of bicomplex pseudoanalytic formal powers. To our best knowledge
this is the first result of this kind in bicomplex pseudoanalytic function
theory. Its importance is in the fact that it opens the way for construction
of Bergman-type reproducing kernels for corresponding second-order elliptic
equations with variable complex-valued coefficients and hence for solving
boundary and eigenvalue problems.

\end{document}